\pdfoutput=1
\RequirePackage{ifpdf}
\ifpdf % We are running pdfTeX in pdf mode
\documentclass[pdftex]{sigma}
\else
\documentclass{sigma}
\fi

\newcommand{\CC}{\mathbb{C}}
\newcommand{\GG}{\mathbb{G}}
\newcommand{\HH}{\mathbb{H}}
\newcommand{\NN}{\mathbb{N}}
\newcommand{\QQ}{\mathbb{Q}}
\newcommand{\RR}{\mathbb{R}}
\newcommand{\ZZ}{\mathbb{Z}}

\newcommand{\Aa}{\mathbf{A}}
\newcommand{\Ga}{\mathbf{G}}
\newcommand{\La}{\mathbf{L}}
\newcommand{\Ma}{\mathbf{M}}
\newcommand{\Na}{\mathbf{N}}
\newcommand{\Pa}{\mathbf{P}}
\newcommand{\Ra}{\mathbf{R}}
\newcommand{\Sa}{\mathbf{S}}
\newcommand{\Ta}{\mathbf{T}}
\newcommand{\Ua}{\mathbf{U}}

\DeclareMathOperator{\Ad}{Ad}
\DeclareMathOperator{\rank}{rank}

\DeclareMathOperator{\SO}{SO}
\DeclareMathOperator{\Sp}{Sp}
\DeclareMathOperator{\SL}{SL}
\DeclareMathOperator{\SU}{SU}
\DeclareMathOperator{\Unitary}{U}
\DeclareMathOperator{\unip}{unip}

\newcommand{\cover}{\widetilde}
\newcommand{\ints}{\mathcal{O}}
\newcommand{\quot}{\overline}
\newcommand{\hyper}{\mathfrak{H}}
\newcommand{\Lie}[1]{\lowercase{\mathfrak{#1}}}

\newcommand{\subgroups}{\mathcal{L}}

\newcommand{\ra}{\overrightarrow}
\newcommand{\la}{\overleftarrow}

\numberwithin{equation}{section}

\newtheorem{thm}{Theorem}[section]
\newtheorem{lem}[thm]{Lemma}
\newtheorem{prop}[thm]{Proposition}
\newtheorem{cor}[thm]{Corollary}

\theoremstyle{definition}
\newtheorem{assumption}[thm]{Assumption}
\newtheorem{defn}[thm]{Definition}
\newtheorem{notation}[thm]{Notation}
\newtheorem{rem}[thm]{Remark}
\newtheorem{rems}[thm]{Remarks}

\begin{document}

\allowdisplaybreaks

\newcommand{\arXivNumber}{1908.02365}

\renewcommand{\PaperNumber}{012}

\FirstPageHeading

\ShortArticleName{Quasi-Isometric Bounded Generation by ${\mathbb Q}$-Rank-One Subgroups}

\ArticleName{Quasi-Isometric Bounded Generation \\ by ${\mathbb Q}$-Rank-One Subgroups}

\Author{Dave Witte MORRIS}

\AuthorNameForHeading{D.W.~Morris}

\Address{Department of Mathematics and Computer Science, University of Lethbridge,\\
Lethbridge, Alberta, T1K~3M4, Canada}
\Email{\href{mailto:Dave.Morris@uleth.ca}{Dave.Morris@uleth.ca}}
\URLaddress{\url{http://people.uleth.ca/~dave.morris/}}

\ArticleDates{Received August 16, 2019, in final form March 05, 2020; Published online March 11, 2020}

\Abstract{We say that a subset~$X$ \emph{quasi-isometrically boundedly generates} a finitely gene\-ra\-ted group~$\Gamma$ if each element~$\gamma$ of a finite-index subgroup of~$\Gamma$ can be written as a~pro\-duct $\gamma = x_1 x_2 \cdots x_r$ of a bounded number of elements of~$X$, such that the word length of each~$x_i$ is bounded by a constant times the word length of~$\gamma$. A.~Lubotzky, S.~Mozes, and M.S.~Raghunathan observed in 1993 that ${\rm SL}(n,{\mathbb Z})$ is quasi-isometrically boundedly generated by the elements of its natural ${\rm SL}(2,{\mathbb Z})$ subgroups. We generalize (a slightly weakened version of) this by showing that every $S$-arithmetic subgroup of an isotropic, almost-simple ${\mathbb Q}$-group is quasi-isometrically boundedly generated by standard ${\mathbb Q}$-rank-1 subgroups.}


\Classification{22E40; 20F65; 11F06}

\section{Introduction} \label{IntroSect}

A subset $X$ ``boundedly generates'' a group~$\Gamma$ if every element~$\gamma$ of~$\Gamma$ can be written as a~product $\gamma = x_1 x_2 \cdots x_r$ of a bounded number of elements of~$X$ (cf.~\cite[p.~203]{MR1278263}). This is a~very powerful notion in abstract group theory, but geometric group theorists may need the bounded generation to be ``quasi-isometric'', which means that the word length of every~$x_i$ is bounded by a~constant times the word length of~$\gamma$ (after passing to a finite-index subgroup). Lubotzky, Mozes, and Raghunathan \cite[Corollary~3]{MR1244421} observed in 1993 that $\SL(n,\ZZ)$ is quasi-isometrically boundedly generated by the elements of its natural $\SL(2,\ZZ)$ subgroups (but did not use this terminology). This implies that every finite-index subgroup~$\Gamma$ of $\SL(n,\ZZ)$ is quasi-isometrically boundedly generated by finitely many subgroups, each of which is the intersection of~$\Gamma$ with a conjugate of a natural $\SL(2,\ZZ)$ subgroup of $\SL(n,\ZZ)$.

We generalize this by showing that every arithmetic subgroup of an isotropic, almost-simple $\QQ$-group is quasi-isometrically boundedly generated by standard $\QQ$-rank-1 subgroups. The proof is a minor modification of an argument of Lubotzky, Mozes, and Raghunathan \cite[Section~4]{MR1828742} that establishes a slightly weaker statement. (One could describe their result by saying that it omits the word ``boundedly'' from the conclusion.)

\begin{defn} \label{QIBddGenDefn}
Let $\Gamma$ be an arithmetic subgroup of a semisimple algebraic $\QQ$-group~$\Ga$. (So~$\Gamma$ is commensurable with $\Ga(\ZZ)$.) We say that $\Gamma$ is \emph{quasi-isometrically boundedly generated by standard $\QQ$-rank-$1$ subgroups} if there exist constants $r = r(\Ga,\Gamma) \in \NN$ and $C = C(\Ga,\Gamma) \in \RR^+$, a~finite subset $\Gamma_0 = \Gamma_0(\Ga,\Gamma)$ of~$\Gamma$, and a~finite collection $\subgroups = \subgroups(\Ga,\Gamma)$ of Zariski closed subgroups of~$\Ga$, such that:
\begin{enumerate}\itemsep=0pt
\item %\label{QIBddGenDefn-product}
	 Every element~$\gamma$ of~$\Gamma$ can be written in the form $\gamma = x_1 x_2 \cdots x_r$, where, for each~$i$, we have either:
\begin{enumerate}\itemsep=0pt
\item %\label{QIBddGenDefn-product-sinL}
$x_i \in \La \cap \Gamma$, for some $\La \in \subgroups$, and $\log \|x_i\| \le C \log \|\gamma\|$, or
\item $x_i \in \Gamma_0$.
\end{enumerate}
\item %\label{QIBddGenDefn-L}
Each $\La \in \subgroups$ is a standard $\QQ$-rank-1 subgroup of~$\Ga$. This means that $\La$ is a connected, almost $\QQ$-simple subgroup of~$\Ga$, with $\rank_\QQ \La = 1$, and that there is a root~$\alpha$ of some maximal $\QQ$-split torus~$\Ta$ of~$\Ga$, such that the Lie algebra of~$\La$ is generated by the root spaces $\Lie G_\beta$ for $\beta \in \big\{{\pm} \alpha, \pm 2\alpha, \pm \frac{1}{2} \alpha \big\}$.
\end{enumerate}
\end{defn}

\begin{thm} \label{QIBddGenByRank1}Every arithmetic subgroup of an isotropic, almost-simple $\QQ$-group is quasi-isometrically boundedly generated by standard $\QQ$-rank-$1$ subgroups.
\end{thm}

\begin{rems} \label{MainRems}
\leavevmode
\begin{enumerate}\itemsep=0pt
\item In the special case where the arithmetic subgroup is $\SL(n,\ZZ)$, the Comptes Rendus note of Lubotzky, Mozes, and Raghunathan \cite[Corollary~3]{MR1244421} not only proves the full strength of our theorem, but also provides the much more explicit conclusion that the collection~$\subgroups$ of standard $\QQ$-rank-1 subgroups can be taken to be the set of natural copies of $\SL(2,\ZZ)$ in $\SL(n,\ZZ)$, that the length~$r$ of the word can be taken to be $n^2 - n$, and that no additional finite set~$\Gamma_0$ is required. (See Proposition~\ref{SLnZThm} below.)
	
\item \label{MainRems-wordlength}Assuming that $\rank_\RR \Ga \ge 2$ (which is the only nontrivial case of Definition~\ref{QIBddGenDefn}), a famous theorem of Lubotzky, Mozes, and Raghunathan \cite[Theorem~A]{MR1828742} tells us that the word length $\ell_\Gamma(\gamma)$ is within a bounded multiplicative constant of $\log \| \gamma \|$. Therefore, the bound on $\log \|x_i\|$ in Definition~\ref{QIBddGenDefn}(1a) tells us that $\sum\limits_{i=1}^r \ell_\Gamma(x_i) \le C \ell_\Gamma(\gamma)$, for some constant~$C$. This is the motivation for the use of the term ``quasi-isometric'' in Definition~\ref{QIBddGenDefn}.
\item See Section~\ref{SarithSect} for extensions of Theorem~\ref{QIBddGenByRank1} to the setting of $S$-arithmetic subgroups.
\end{enumerate}
\end{rems}

The Margulis arithmeticity theorem (cf.\ \cite[Theorem~5.2.1, p.~90]{Morris-IntroArithGrps}) implies that Theorem~\ref{QIBddGenByRank1} can also be stated in the language of Lie groups:

\begin{cor} \label{QIBddGenLie}
Let $\Gamma$ be a noncocompact, irreducible lattice in a connected, semisimple Lie group~$G$ with finite center. Then $\Gamma$ is quasi-isometrically boundedly generated by $\QQ$-rank-$1$ subgroups. More precisely, there exist constants $r = r(\Ga,\Gamma) \in \NN$ and $C = C(\Ga,\Gamma) \in \RR^+$, a finite subset $\Gamma_0 = \Gamma_0(G,\Gamma)$ of~$\Gamma$, and a finite collection $\subgroups = \subgroups(G,\Gamma)$ of closed, connected, semisimple subgroups of~$G$, such that:
	\begin{enumerate}\itemsep=0pt
\item[$1.$] %\label{QIBddGenLie-product}
	 Every element~$\gamma$ of~$\Gamma$ can be written in the form $\gamma = x_1 x_2 \cdots x_r$, where, for each~$i$, we have either:
\begin{enumerate}\itemsep=0pt
\item[$(a)$] %\label{QIBddGenLie-product-sinL}
		$x_i \in L \cap \Gamma$, for some $L \in \subgroups$, and $\log \|x_i\| \le C \log \|\gamma\|$,
		or
\item[$(b)$] $x_i \in \Gamma_0$.
		\end{enumerate}
\item[$2.$] For each $L \in \subgroups$, the intersection $L \cap \Gamma$ is an irreducible lattice in~$L$, with $\rank_\QQ (L \cap \Gamma) = 1$.

	\end{enumerate}
\end{cor}

\begin{rems}
In the statement of the corollary:
\begin{enumerate}\itemsep=0pt
	\item The notion of $\QQ$-rank is the natural extension that applies to all lattices in semisimple Lie groups, not only the arithmetic ones: non-arithmetic irreducible lattices arise only in real rank one, where the $\QQ$-rank is defined to be~$0$ for cocompact lattices, and to be~$1$ for non-cocompact lattices (see \cite[Definition~9.1.8, pp.~193--194]{Morris-IntroArithGrps}).
	\item $\| x \|$ can be taken to be the norm of~$x$ in a finite-dimensional matrix representation of~$G$ that has finite kernel, or $\log \|x\|$ could be replaced with either the word length $\ell_\Gamma(x)$, or the distance from the identity to~$x$, with respect to a left-invariant Riemannian metric on~$G$.
	\item Each subgroup in $\subgroups(G,\Gamma)$ may be assumed to be ``standard'', as in Dedinition~\ref{QIBddGenDefn}(2). This condition is a bit complicated to state in Lie-theoretic terms, but see Definition~\ref{StandardQInfiniteCenterDefn}.
	\item The assumption that $G$ has finite center can be eliminated (see Corollary~\ref{InfCenterBddGen}).
	\end{enumerate}
\end{rems}

\section[Idea of the proof: the example of $\SL(n,\pmb\ZZ)$]{Idea of the proof: the example of $\boldsymbol{\SL(n,\pmb\ZZ)}$} \label{SLnZSect}

This section is purely expository, and may therefore be passed over. It presents the proof in the (known) easy special case where $\Gamma = \SL(n,\ZZ)$, and then describes the main issues that are involved in generalizing the argument.

\begin{notation} \label{PrecNotation}
For $s,t \in \RR^+$, we write $s \prec t$ if $s$ is bounded by a polynomial function of~$t$: $s \le t^C + C$ for some constant $C > 0$. (The constant must be independent of $s$ and~$t$, but may depend on the parameter~$n$ in $\SL(n,\ZZ)$.) Equivalently, $\log s \le \max\{C' \log t, C'\}$, for some constant $C' > 0$.
\end{notation}

\begin{notation}
Let $L_{k,\ell}$ be the copy of $\SL(2,\ZZ)$ in $\SL(n,\pmb\ZZ)$ that is supported on the matrix entries $(k,k)$, $(k,\ell)$, $(\ell,k)$, and $(\ell,\ell)$.
\end{notation}

\begin{prop}[Lubotzky--Mozes--Raghunathan {\cite[Corollary~3]{MR1244421}}] \label{SLnZThm}
Let $r = n^2 - n$. For all $\gamma \in \SL(n, \ZZ)$, there exist $x_1,\ldots,x_r \in \SL(n, \ZZ)$, such that:
\begin{enumerate}\itemsep=0pt
\item[$1)$] $\gamma = x_1 x_2 \cdots x_r$,
\item[$2)$] for each~$i$, there exist $k,\ell$, such that $x_i \in L_{k,\ell}$,
	and
\item[$3)$] $\|x_i\| \prec \|\gamma\|$ for all~$i$.
	\end{enumerate}
\end{prop}

\begin{proof}[Proof \normalfont (Lubotzky--Mozes--Raghunathan)]
For each nonzero $v \in \ZZ^2$, it is easy to find $x \in \SL(2,\ZZ)$, such that $xv = \left[\begin{smallmatrix} + \\ 0 \end{smallmatrix}\right]$, and $\| x \| \prec \| v \|$.
	This implies there is $x_{2,1} \in L_{2,1}$, such that $(x_{2,1} \gamma)_{2,1} = 0$.
	Then, similarly, there is $x_{3,1} \in L_{3,1}$, such that $(x_{3,1} x_{2,1} \gamma)_{3,1} = 0$. It is important to note that multiplying on the left by $x_{3,1}$ does not affect the second row (indeed, it only affects the 1st row and the 3rd row), so we have $(x_{3,1} x_{2,1} \gamma)_{2,1} = (x_{2,1} \gamma)_{2,1} = 0$. Continuing in this way, we have $x_{4,1}, x_{5,1}, \ldots, x_{n,1}$, so that if we let
\[\gamma^{(1)} = x_{n,1} x_{n-1,1} \cdots x_{2,1} \gamma ,\]
then $\gamma^{(1)}_{i,1} = 0$ for $i = 2,3,\ldots, n$.
(We may also assume $\gamma^{(1)}_{1,1} > 0$. Since $\gamma^{(1)} \in \SL(n, \ZZ)$, it is easy to see that this implies $\gamma^{(1)}_{1,1} = 1$.)

Next, we work on the second column: for $i = 3,4, \ldots, n$, there exists $x_{i,2} \in L_{i,2}$, such that if $\gamma^{(2)} = x_{n,2} x_{n-1,2} \cdots x_{3,2} \gamma^{(1)}$, then $\gamma^{(2)}_{i,2} = 0$ for $i = 3,4,\ldots,n$. (Also, $\gamma^{(2)}_{2,2} = 1$.)
Continuing in this way, column-after-column, we obtain $\gamma^{(n-1)}$, such that $\gamma^{(n-1)}_{i,j} = 0$ for $i > j$. Also, $\gamma^{(n-1)}_{i,i} =1$ for all~$i$.

So $\gamma^{(n-1)}$ is a unipotent upper-triangular matrix. It is then easy to reduce to the identity matrix by multiplying by unipotent upper-triangular elements of the various $L_{k,\ell}$.
\end{proof}

Extending these ideas to a proof of Theorem~\ref{QIBddGenByRank1} encounters a few complications. All but the first are dealt with by the argument of Lubotzky, Mozes, and Raghunathan in \cite[Section~4]{MR1828742}.

\subsection*{The first complication}
In order to retain the progress that has been made in earlier steps, the matrix entries need to be annihilated in a carefully chosen order.
For example, suppose we have $\gamma_{2,1} = 0$. There exists $x \in L_{2,3}$, such that $(x \gamma)_{3,2} = 0$. However, it will probably not be the case that $(x \gamma)_{2,1} = 0$. So we should not annihilate the $(3,2)$-entry until after the $(3,1)$-entry has been annihilated.

Also, there may not be a convenient matrix representation of~$\Gamma$. However, the off-diagonal entries of matrices in $\SL(n,\ZZ)$ correspond to roots of the algebraic group, so, instead of annihi\-la\-ting ``matrix entries'', we will instead write a (generic) element of~$\Gamma$ in the form $\gamma = \bigl(\prod\limits_{\alpha\in \Phi^+} u_\alpha \bigr) \cdot p$, where $u_\alpha$ is in the root subgroup of~$\Ga$ corresponding to the positive root~$\alpha$, and $p$ is in a minimal parabolic $\QQ$-subgroup.
The product is taken with respect to an appropriately chosen order of the positive roots, and we will annihilate each~$u_\alpha$ one-by-one, in this order, by multiplying by an element of the standard $\QQ$-rank-1 subgroup~$\La_\alpha$ that corresponds to~$\alpha$.

When $\rank_\QQ \Ga = 2$, the root system is in~$\RR^2$, and we will use the clockwise ordering of the positive roots (or, in other words, left-to-right, if the positive roots are in the upper half-plane). For higher rank, we project the root system to a generic 2-dimensional plane, and then use the clockwise order in this 2-dimensional representation.

\begin{rem}\looseness=-1 The argument of Lubotzky, Mozes, and Raghunathan in \cite[Section~4]{MR1828742} annihilates the root elements in a different order; it allows the annihilation of a particular~$u_\alpha$ to have side effects that undermine some of the previous work. Thus, the inductive argument re-annihilates some roots an unbounded number of times. The only new idea in the present paper is the observation that the roots can be annihilated in an order that never requires a root to annihilated more than once.
\end{rem}

\subsection*{The second complication}
\looseness=-1 A general arithmetic subgroup of $\QQ$-rank one may have more than one cusp. This means that multiplying by an element of $\La_\alpha \cap \Gamma$ might not be able to annihilate $u_\alpha$. But there are only finitely many cusps, which means that $u_\alpha$ can be moved into a finite set. It turns out that each element of this set simply applies a conjugation to the calculations. Therefore, to deal with this issue, it suffices to allow (finitely many) conjugates of the $\QQ$-rank-one subgroups that we started with.

\subsection*{The third complication}
The above arguments reduce to the case where $\gamma$ is in a minimal parabolic $\QQ$-subgroup~$\Pa$. In $\SL(n,\ZZ)$, this implies (up to finite index) that $\gamma$ is in the unipotent radical of~$\Pa$, and it is easy to finish the proof from there. In general, however, the anisotropic part~$\Ma$ of the reductive Levi factor of~$\Pa$ may have infinitely many integer points.

Fortunately, this is a very minor issue. To deal with it, Lemma~\ref{LeviOK} points out that if $\rank_\QQ \Ga \ge 2$, then there are finitely many standard $\QQ$-subgroups $\La_1,\ldots,\La_k$ of~$\Ga$, such that $\rank_\QQ \La_i < \rank_\QQ \Ga$ for all~$i$, and (up to finite index) every element of $\Ma(\ZZ)$ can be efficiently written as the product of a bounded number of elements of $\La_1(\ZZ) \cup \La_2(\ZZ) \cup\cdots \cup \La_k(\ZZ)$. By induction on $\rank_\QQ \Ga$, we may assume that each~$\La_i$ is quasi-isometrically boundedly generated by $\QQ$-rank-1 subgroups. It is then easy to complete the proof.

\section{Preliminaries}

 We assume familiarity with the basic theory of algebraic groups over~$\QQ$.

 \begin{notation} \label{GaNotation} The symbol~$\Ga$ always denotes a simply connected, $\QQ$-isotropic, almost $\QQ$-simple algebraic group (except that $\Ga$ is not assumed to be simply connected in Lemma~\ref{InvtUnderIsogeny}, and $\QQ$~is replaced by a more general number field in Proposition~\ref{GeneralSArith}).
 \end{notation}

Although Definition~\ref{QIBddGenDefn} refers to a subgroup~$\Gamma$ of~$\Ga$, the following observation shows that it is actually a property of the $\QQ$-group~$\Ga$, rather than a property of the arithmetic group~$\Gamma$ (or the pair~$(\Ga,\Gamma)$).

\begin{lem} \label{BddCommensurable}\sloppy
Quasi-isometric bounded generation by $\QQ$-rank-$1$ subgroups is invariant under commensurability:
if $\Gamma_1$ is quasi-isometrically boundedly generated by $\QQ$-rank-$1$ subgroups, and~$\Gamma_2$ is a subgroup of~$\Ga$ that is commensurable to~$\Gamma_1$, then $\Gamma_2$ also has this property $($but perhaps with a different choice of the constants~$r$ and~$C$, the finite set~$\Gamma_0$, and the collection~$\subgroups)$.
\end{lem}

\begin{proof}
We may assume that $\Gamma_2$ is either a finite-index subgroup of~$\Gamma_1$, or a~finite-index supergroup of~$\Gamma_1$. We apply standard arguments from geometric group theory to each of these cases.

If $\Gamma_2$ is a finite-index supergroup of~$\Gamma_1$, let $\Lambda_0$ be a set of coset representatives of $\Gamma_2 / \Gamma_1$. Then every element of $\Gamma_2$ is of the form $\lambda x_1 x_2 \cdots x_r$, with $\lambda \in \Lambda_0 \subseteq \Gamma_0(\Ga,\Gamma_2)$, and either $x_i \in \Gamma_0(\Ga,\Gamma_1)$ or $x_i \in \La \cap \Gamma_1 \subseteq \La \cap \Gamma_2$, with $\La \in \subgroups(\Ga,\Gamma_1)$. Therefore, we may let $\Gamma_0(\Ga,\Gamma_2) = \Gamma_0(\Ga,\Gamma_1) \cup \Lambda_0$,
$\subgroups(\Ga,\Gamma_2) = \subgroups(\Ga,\Gamma_1)$ and $r(\Ga,\Gamma_2) = r(\Ga,\Gamma_1) + 1$.

We now consider the possibility that $\Gamma_2$ is a~finite-index subgroup of~$\Gamma_1$. By passing to a~further subgroup, we may assume that $\Gamma_2$ is normal in~$\Gamma_1$ (and has finite index).
Let $\Lambda_0$ be a symmetric finite subset of~$\Gamma_1$ that contains a set of coset representatives of $\Gamma_1/\Gamma_2$ (and also contains the identity element). We know that every element~$\gamma$ of~$\Gamma_2$ is of the form $\gamma = x_1 x_2 \cdots x_r$, with each $x_i$ either in some $\La \cap \Gamma_1$ or in $\Gamma_0(\Ga,\Gamma_1)$. By enlarging $\Gamma_0(\Ga,\Gamma_1)$ to include a set of coset representatives of each $(\La \cap \Gamma_1)/(\La \cap \Gamma_2)$ (and doubling~$r$), we may assume each $x_i$ is either in some $\La \cap \Gamma_2$ or in $\Gamma_0(\Ga,\Gamma_1)$.
Let
\begin{gather*}
\Gamma_0(\Ga,\Gamma_2) = \bigl( \Lambda_0 \cdot \Gamma_0(\Ga,\Gamma_1) \cdot \Lambda_0 \bigr) \cap \Gamma_2,\\
\subgroups(\Ga,\Gamma_2) = \big\{ \lambda \La \lambda^{-1} \,|\, \La \in \subgroups(\Ga,\Gamma_1), \, \lambda \in \Lambda_0 \big\} .
\end{gather*}
For each $i$, there exists $\lambda_i \in \Lambda_0$, such that $x_1 x_2 \cdots x_i \lambda_i \in \Gamma_2$. Note that if $x_i \notin \Gamma_0(\Ga,\Gamma_1)$, then we have
	\[ x_1 x_2 \cdots x_i \lambda_{i-1}
	= x_1 x_2 \cdots x_{i-1} \lambda_{i-1} \cdot \lambda_{i-1}^{-1} x_i \lambda_{i-1}
	\in \Gamma_2 \cdot \lambda_{i-1}^{-1} \Gamma_2 \lambda_{i-1}
	= \Gamma_2
	\]
(since $\Gamma_2 \triangleleft \Gamma_1$). Therefore, we may assume $\lambda_i = \lambda_{i-1}$ whenever $x_i \notin \Gamma_0(\Ga,\Gamma_1)$. We may also assume $\lambda_r = 1$, since $x_1 x_2 \cdots x_r = \gamma \in \Gamma_2$.
Define
	\[ x_i' = \lambda_{i-1}^{-1} x_i \lambda_i
	\qquad \text{(where $\lambda_0 = 1$)} .\]
Then each $x_i'$ is either in $\La \cap \Gamma_2$ for some $\La \in \subgroups(\Ga,\Gamma_2)$, or in $\Gamma_0(\Ga,\Gamma_2)$,
	and we have
		\begin{gather*}
		\gamma
		 = x_1 x_2 \cdots x_r
		= \big(\lambda_0^{-1} x_1 \lambda_1\big) \big(\lambda_1^{-1} x_1 \lambda_2\big)
		\cdots \big(\lambda_{r-1}^{-1} x_r \lambda_r\big) \lambda_r^{-1}
		= x_1' x_2' \cdots x_r'
		 \end{gather*}
since $\lambda_r = 1$.
\end{proof}

\begin{rem}The above proof shows that we may choose every element of $\subgroups(\Ga,\Gamma_2)$ to be $\Gamma_1$-conjugate to an element of $\subgroups(\Ga,\Gamma_1)$.
\end{rem}

Since the image of an arithmetic subgroup under an isogeny is an arithmetic subgroup of the image, a similar argument yields the following result (which explains why we may assume that~$\Ga$ is simply connected in Notation~\ref{GaNotation}):

\begin{lem} \label{InvtUnderIsogeny}
Quasi-isometric bounded generation by $\QQ$-rank-$1$ subgroups is invariant under isogeny:
the arithmetic subgroups of~$\Ga$ are quasi-isometrically boundedly generated by $\QQ$-rank-$1$ subgroups if and only if the arithmetic subgroups of the universal cover~$\cover\Ga$ are quasi-isometrically boundedly generated by $\QQ$-rank-$1$ subgroups.
\end{lem}

Let $\Gamma$ be an arithmetic subgroup of~$\Ga$. In order to prove Theorem~\ref{QIBddGenByRank1}, we wish to show that~$\Gamma$ is quasi-isometrically boundedly generated by standard $\QQ$-rank-1 subgroups.

\begin{notation} \label{QIBddGenNotation}\leavevmode
	\begin{enumerate}\itemsep=0pt
\item %\label{QIBddGenNotation-G(Z)}
	We assume $\Gamma = \Ga(\ZZ)$ (see Lemma~\ref{BddCommensurable}).
\item %\label{QIBddGenNotation-T}
	Let $\Ta$ be a maximal $\QQ$-split torus of~$\Ga$. We may assume $\dim \Ta \ge 2$ (in other words, $\rank_\QQ \Ga \ge 2$), for otherwise it is obvious that $\Gamma$ has the desired property (by letting $\subgroups = \{\Ga\}$, $r = C = 1$, and $\Gamma_0 = \varnothing$).
\item Let $\Phi$ be the $\QQ$-root system of~$\Ga$ corresponding to~$\Ta$.
	(So $\Phi$ is a subset of the character group~$\Ta^*$.)
\item Let $\eta \colon \Ta^* \otimes \RR \to \CC$ be an $\RR$-linear map, such that $\eta(\Phi)$ is disjoint from~$\RR$, and $\eta$~is injective on the set $\RR \Phi$, which is a finite union of lines through the origin. (These conditions hold for a generic choice of~$\eta$.)
\item Let $\Phi^+ = \Phi \cap \eta^{-1}(\hyper)$, where $\hyper \subset \CC$ is the (open) upper half-plane.
Since $\eta^{-1}(\hyper)$ is an open half-space in $\Ta^* \otimes \RR$ (and its boundary $\eta^{-1}(\RR)$ is disjoint from~$\Phi$), we know that $\Phi^+$ is a choice of positive roots for~$\Phi$.
\item Let $\Phi^+_1, \Phi^+_2,\ldots,\Phi^+_k$ be a list of the equivalence classes of positive roots, where two positive roots are equivalent if they are scalar multiples of each other. We specify the ordering of this list by requiring that
		\[ \text{the rays $\RR^+ \eta(\Phi^+_1)$, $\RR^+ \eta(\Phi^+_2), \ldots$, $\RR^+ \eta(\Phi^+_k)$ are in clockwise order} \]
	in the upper half-plane.
	(By convention, we let $\RR^+ \eta(\Phi^+_0) = \RR^-$ and $\RR^+ \eta(\Phi^+_{k+1}) = \RR^+$, so $\Phi^+_0 = \Phi^+_{k+1} = \varnothing$.)
	Let $\Phi_i = \Phi^+_i \cup - \Phi^+_i$ for $0 \le i \le k + 1$.
	
\item For each $i \in \{0,1,2,\ldots,k+1\}$, the difference $\CC \smallsetminus \RR \eta(\Phi^+_i)$ is the union of two open half-planes.
\begin{itemize}\itemsep=0pt
\item The half-plane on the \emph{right} of $\eta(\Phi^+_i)$ contains $\eta(\Phi^+_j)$ for $j > i$ (unless $i = 0$ and $j = k+1$), whereas
\item the half-plane on the \emph{left} contains $\eta(\Phi^+_j)$ for $j < i$ (unless $i = k + 1$ and $j = 0$).
\end{itemize}
We say that a root~$\phi$ is to the \emph{left} of~$\Phi_i$ if $\eta(\phi)$ is in the half-plane that is on the left of~$\eta(\Phi^+_i)$; conversely, $\phi$ is to the \emph{right} of~$\Phi_i$ if $\eta(\phi)$ is in the half-plane that is on the right of~$\eta(\Phi^+_i)$
		\begin{alignat*}{3}
		& \ra{\Phi_i}= \{ \phi \in \Phi \,|\, \text{$\phi$ is to the right of $\Phi^+_i$} \},
		\qquad &&
		\ra{\Phi^+_i}= \ra{\Phi_i} \cap \Phi^+,& \\
		&\la{\Phi_i}= \{ \phi \in \Phi \,|\, \text{$\phi$ is to the left of $\Phi^+_i$} \},
		\qquad &&
		\la{\Phi^+_i}= \la{\Phi_i} \cap \Phi^+		.&
\end{alignat*}
Note that each of these is a closed set of roots. (That is, if $\alpha$ and~$\beta$ are elements of one of these sets, and $\alpha + \beta \in \Phi$, then $\alpha + \beta$ is also in that same set.)
Furthermore, $\Phi_i^+ \cup \ra{\Phi_i}$ is a~choice of positive roots for~$\Phi$, for $1 \le i \le k$.
Also note that $\la{\Phi^+_{i+1}}$ is the disjoint union of~$\la{\Phi^+_i}$ and~$\Phi^+_i$, for $1 \le i \le k$, and that
	 $\la{\Phi^+_1} = \ra{\Phi^+_k} = \varnothing$ and $\ra{\Phi^+_0} = \la{\Phi^+_{k+1}} = \Phi^+$.
\item For $i = 1,2,\ldots,k$, let $\la{\Ua_i^+}$, $\Ua_i^+$, $\ra{\Ua_i^+}$, and $\ra{\Ua_i}$ be the (unipotent) $\QQ$-subgroups of~$\Ga$ whose Lie algebras are spanned by the root spaces corresponding to the roots in $\la{\Phi_i^+}$, $\Phi_i^+$, $\ra{\Phi_i^+}$, and $\ra{\Phi_i}$, respectively.
\item Let $\Ua^+$ be the maximal unipotent $\QQ$-subgroup of~$\Ga$ that corresponds to~$\Phi^+$, so we have $\Ua^+ = \ra{\Ua_i^+} \Ua^+_i \la{\Ua_i^+}$, and the corresponding product map $\ra{\Ua_i^+} \times \Ua^+_i \times \la{\Ua_i^+} \to \Ua^+$ is an isomorphism of varieties, for $1 \le i \le k$ \cite[Proposition~3.11, pp.~77--78]{BorelTits-GrpRed}.
	
\item Let $\Pa^-$ be the minimal parabolic $\QQ$-subgroup of~$\Ga$ that contains~$\Ta$ and is \emph{opposite} to~$\Ua^+$ (so its unipotent radical $\unip \Pa^-$ is spanned by the roots in $-\Phi^+$).

\item For $i \in \{1,2,\ldots,k\}$, we let $\Ga_i$ be the standard $\QQ$-rank-one subgroup whose Lie algebra is generated by the root spaces corresponding to roots in $\Phi_i$. Note that $\Ga_i$ has a parabolic $\QQ$-subgroup $\Pa_i^- = \Pa^- \cap \Ga_i$ that is opposite to~$\Ua_i^+$.

\item Let $\Omega = \Ua^+ \Pa^-$, so $\Omega$ is a Zariski-open, dense subset of~$\Ga$, and the natural product map $\Ua^+ \times \Pa^- \to \Omega$ is an isomorphism of varieties \cite[Proposition~4.10(d), pp.~88--89]{BorelTits-GrpRed}. For $g \in \Omega(\QQ)$, we may write
		\[ g = u^+(g) p^-(g) \qquad \text{with} \quad u^+(g) \in \Ua^+(\QQ) \quad \text{and} \quad p^-(g) \in \Pa^-(\QQ) . \]
Furthermore, we write
	\[ u^+(g) = \la{u^+_i}(g) u^+_i(g) \ra{u^+_i}(g) , \qquad \text{with} \quad \la{u^+_i}(g) \in \la{\Ua_i^+}, \quad u^+_i(g) \in \Ua_i^+, \quad \ra{u^+_i}(g) \in \ra{\Ua_i^+} .\]
	
	\item As mentioned in Notation~\ref{PrecNotation}, for $s,t \in \RR^+$, we write $s \prec t$ when there is a constant~$C$, such that $\log s \le \max\{C, C \log t\}$. The constant~$C$ may depend on~$\Ga$, and may also depend on the choices made above (such as the $\QQ$-split torus~$\Ta$ and the $\RR$-linear map~$\eta$), and may also depend on choices made the course of the proofs below (such as the dominant weight~$\lambda_\alpha$ in Definition~\ref{omegaDefn} and the finite-index subgroup~$\Gamma_i$ and finite set~$F_{i+1}$ in Lemma~\ref{OneStep}). Of course, however, $C$~cannot depend on $s$ and~$t$.
\end{enumerate}
\end{notation}

\begin{assumption} \label{AbsSimple}Theorem~\ref{QIBddGenByRank1} may be of interest to non-algebraists, so, to eliminate some algebraic technicalities, we will henceforth assume that~$\Ga$ is absolutely almost simple. This means we are assuming that the Lie group $G = \Ga(\RR)$ is almost simple and is not a complex Lie group. For example, $G$ can be $\SL(n,\RR)$ or $\SO(p,q)$ (with $p + q \ge 5$) or $\SU(p,q)$ or $\Sp(2n,\RR)$ or $\Sp(p,q)$ (but $G$ cannot be a group such as $\SL(n,\RR) \times \SL(n,\RR)$ that has more than one simple factor, and cannot be a complex group, such as $\SL(n,\CC)$ or $\SO(n,\CC)$).
(Actually, $\SO(p,q)$ is not simply connected as an algebraic group, so, although the theorem applies to it, the proof needs to consider the corresponding spin group, which is a double cover.)
See Remark~\ref{ChangeForK} for minor modifications of the proof that remove this restriction.
\end{assumption}

To prove Theorem~\ref{QIBddGenByRank1}, we need to show that each $\gamma \in \Gamma$ can be written as an appropriate product $\gamma = x_1x_2\cdots x_r$. The upshot of the following easy lemma (combined with induction on $\rank_\QQ \Ga$) is that it will suffice to write $\gamma = x_1x_2\cdots x_r m$, where $m$~is in a Levi factor of some member of a finite collection of proper parabolic $\QQ$-subgroups of~$\Ga$.

\begin{lem} \label{LeviOK}
Recall that $\rank_\QQ \Ga \ge 2$ $($see Notation~{\rm \ref{QIBddGenNotation}(2))}.
Let $\Ma$ be a $($reductive$)$ Levi factor of some proper parabolic $\QQ$-subgroup of~$\Ga$. Then there is a finite set~$\subgroups$ of subgroups of~$\Ga$, a~finite subset~$\Gamma_0$ of~$\Ma(\ZZ)$, and $r \in \NN$, such that
\begin{enumerate}\itemsep=0pt
\item[$1)$] each $\La \in \subgroups$ is an isotropic almost-simple factor of a Levi factor of a parabolic $\QQ$-subgroup of~$\Ga$, such that $1 \le \rank_\QQ \La < \rank_\QQ \Ga$, and
\item[$2)$] every element~$\gamma$ of $\Ma(\ZZ)$ can be written in the form $\gamma = x_1 x_2 \cdots x_r$, where, for each~$i$, we have either:
\begin{itemize}\itemsep=0pt
\item $x_i \in \Gamma_0$, or
\item $x_i \in \La(\ZZ)$, for some $\La \in \subgroups$, and $\|x_i\| \prec \|\gamma\|$.
\end{itemize}
\end{enumerate}
\end{lem}

\begin{proof}
Let $\Sa$ be a maximal $\QQ$-torus of~$\Ma$ that contains a maximal $\QQ$-split torus~$\Ta$ of~$\Ma$, and let $\subgroups$ be the set of all subgroups that are an isotropic almost $\QQ$-simple factor of the Levi factor of a maximal parabolic $\QQ$-subgroup of~$\Ga$ that contains~$\Sa$. (The Levi factor is also required to contain~$\Sa$.) Note that $\subgroups$ is finite, since each element of~$\subgroups$ is generated by root subgroups of the maximal torus~$\Sa$.

By assumption, $\Ma$ is a Levi factor of some parabolic $\QQ$-subgroup~$\Pa$ of~$\Ga$. By enlarging~$\Ma$, we may assume $\Pa$ is maximal, so $\Pa$ corresponds to some circled vertex in the Tits index~$\mathcal{D}$ of~$\Ga$. (More precisely, it corresponds to the orbit of this circled vertex under the $*$-action of the Galois group.)
This circled vertex represents a simple root~$\beta$ of~$\Sa$ that is nontrivial on the maximal $\QQ$-split torus~$\Ta$. Since $\rank_\QQ \Ga > 1$, there is some other circled simple root~$\phi$ (that is not in the same orbit of the $*$-action of the Galois group).

We can write $\Ma$ as an almost-direct product $\Ma = \Sa_0 \Ma_1 \Ma_2 \cdots \Ma_k$, where $\Sa_0$~is the central torus, and each $\Ma_i$ is almost $\QQ$-simple.
Then the subgroup $\Sa_0(\ZZ) \Ma_1(\ZZ) \Ma_2(\ZZ) \cdots \Ma_k(\ZZ)$ has finite index in $\Ma(\ZZ)$, so there is no harm in assuming that $\gamma$ is in either $\Sa_0$ or some~$\Ma_i$.

Consider first the case where $\gamma \in \Ma_i$.
Let $\quot{\mathcal{D}}$ be the quotient of~$\mathcal{D}$ in which two vertices are identified if they are in the same orbit of the $*$-action. Then the simple factor $\Ma_i$ is generated by the roots corresponding to the vertices in some component~$C$ of $\quot{\mathcal{D}} \smallsetminus \quot{\beta}$ (cf.\ \cite[????~2.5.4, p.~41]{Tits-Classif}).
And we may assume $\Ma_i$ is anisotropic. (Otherwise, we have $\Ma_i \in \subgroups$, so the desired conclusion is obvious.) This means that $C$ has no circled vertices, so, in particular, the circled vertex $\quot{\phi}$ cannot be in~$C$. Therefore $C \cup \{\quot{\beta}\}$ is contained in a single connected component~$C_1$ of $\quot{\mathcal{D}} \smallsetminus \{\quot{\phi}\}$. Note that $C_1$ corresponds to an almost-simple factor of the Levi factor of the maximal parabolic that corresponds to (the $*$-orbit of) the circled root~$\phi$. This simple factor contains~$\Ma_i$ (because $C_1$ contains~$C$). And this simple factor is isotropic (because $C_1$ contains a circled vertex, namely~$\quot{\beta}$), so it is in~$\subgroups$.

We may now assume $\gamma \in \Sa_0$, so $\gamma$ is in the maximal torus~$\Sa$.
For each root~$\alpha$ in a base~$\Sigma$ of the root system~$\Phi(\Ta, \Ga)$, let
	\[\Sa_\alpha = \bigl( \langle \Ua_\alpha, \Ua_{-\alpha} \rangle \cap \Sa \bigr)^\circ , \]
so $\Sa_\alpha$ is a subtorus of~$\Sa$ that is defined over~$\QQ$. Since $\Ga$ has no anisotropic factors (indeed, $\Ga$ is $\QQ$-simple and isotropic), it is not difficult to see that
	$\langle \Sa_\alpha \,|\, \alpha \in \Sigma \rangle = \Sa$,
so
	$\langle \Sa_\alpha^a \,|\, \alpha \in \Sigma \rangle = \Sa^a$,
where $\Ra^a$ denotes the anisotropic part of a $\QQ$-torus~$\Ra$.
This implies that
	$\langle \Sa_\alpha^a(\ZZ) \,|\, \alpha \in \Sigma \rangle$ is a~finite-index subgroup of~$\Sa^a(\ZZ)$.
Since $\Sa^a(\ZZ)$ is a~finitely generated abelian group, we conclude that each $\gamma$ in a certain finite-index subgroup of $\Sa^a(\ZZ)$ can be written in the form
	$\gamma = \prod\limits_{\alpha \in \Sigma} \gamma_\alpha$,
with $\gamma_\alpha \in \Sa_\alpha^a(\ZZ)$ and $\ell(\gamma_\alpha) < C \ell(\gamma)$, where $\ell(\gamma)$ denotes the word-length of~$\gamma$ in $\Sa^a(\ZZ)$ and $C$ is a~constant. Since $\Sa$ is a torus, it is clear that $\ell(\gamma)$ is coarsely Lipschitz equivalent to $\log \|\gamma\|$. And we know (for example, from the above argument for the case where $\gamma \in \Ma_i$) that $\langle \Ua_\alpha, \Ua_{-\alpha} \rangle$ is contained in some element of~$\subgroups$. This completes the proof.
\end{proof}

\begin{defn}[cf.\ {\cite[p.~44]{MR1828742}}] \label{omegaDefn}
For each simple root~$\alpha$ of~$\Phi^+$, let $\lambda_\alpha$ be a nonzero, dominant, integral $\QQ$-weight of~$\Ga$ that is orthogonal to all of the simple roots of~$\Phi^+$, other than~$\alpha$. (This determines~$\lambda_\alpha$ up to a positive multiple: the condition determines $\lambda_\alpha$ on~$\Ta$, up to a positive multiple, and $\lambda_\alpha$ must be trivial on the anisotropic part of a chosen maximal $\QQ$-torus that contains~$\Ta$, because $\lambda_\alpha$ is defined over~$\QQ$.) After replacing $\lambda_\alpha$ by a positive integral multiple, we may assume there is a finite-dimensional irreducible $\QQ$-representation~$V_\alpha$ of~$\Ga$ whose lowest weight is~$-\lambda_\alpha$, and such that the lowest weight space $V_\alpha^{-\lambda_\alpha}$ is $1$-dimensional. (Indeed, $\lambda_\alpha$~is a fixed point of the $*$-action of the Galois group, so \cite[Theorem~3.3]{MR0277536} implies that it suffices to have $\lambda_\alpha$ be in the root lattice.)
Let $V_\alpha^*$ be the dual of~$V_\alpha$. Choose nonzero $\QQ$-vectors $v_\alpha$ in the lowest weight space of~$V_\alpha$, and~$v_\alpha^*$ in the highest weight space of~$V_\alpha^*$, and define a regular function~$\omega_\alpha$ on~$\Ga$ by
	\[ \omega_\alpha(g) = v_\alpha^*( g v_\alpha) .\]
Since $V_\alpha$ is a $\QQ$-representation, and $\Gamma = \Ga(\ZZ)$, we may assume, by passing to an integral multiple of~$v_\alpha$, that
	\begin{gather} \label{omega(Gamma)}
	 \omega_\alpha(\Gamma) \subseteq \ZZ
	. \end{gather}
\end{defn}

The following result must be well known, but we do not know of a reference for the statement in this generality.

\begin{lem}[cf.\ {\cite[Lemma~4.11(i)]{MR1828742}}] \label{OmegaIdeal}
We have $\Omega = \{ g \in \Ga \,|\, \text{$\omega_\alpha(g) \neq 0$ for all~$\alpha$} \} $.
\end{lem}

\begin{proof}
For each simple root~$\alpha$ of~$\Phi^+$, let $\pi^-_\alpha$ and~$\pi^+_\alpha$ be the $\Ta$-invariant projections of~$V_\alpha$ onto its lowest weight space and its highest weight space, respectively. (So the kernel of~$\pi^\pm_\alpha$ is the sum of the other weight spaces.)

Let $g \in \Omega$, and let $\alpha$ be any simple root of~$\Phi^+$. Since the lowest weight space of~$V_\alpha$ is $\Pa^-$-invariant and $1$-dimensional, we have $p^-(g) v_\alpha = m v_\alpha$, for some nonzero scalar~$m$. Also, since $u^+(g) \in \Ua^+$, we know that $\pi^-_\alpha \bigl( u^+(g) v_\alpha \bigr) = v_\alpha$. Putting these together, we see that
	\[ \pi^-_\alpha(gv_\alpha)
	= \pi^-_\alpha \bigl( u^+(g) p^-(g) v_\alpha \bigr)
	= \pi^-_\alpha \bigl( m v_\alpha \bigr)
	= m v_\alpha, \]
so $\omega_\alpha(g) = v_\alpha^*( g v_\alpha) = m v_\alpha^*( v_\alpha) \neq 0$.

Conversely, suppose $\omega_\alpha(g) \neq 0$, for all~$\alpha$.
Then $\pi^-_\alpha(g v_\alpha) \neq 0$.
Let $w$ be the longest element of the $\QQ$-Weyl group of~$\Ga$. (This means that $w$ sends every positive $\QQ$-root to a negative $\QQ$-root, and vice-versa.)
Then $w$ interchanges the lowest weight with the highest weight, so $\pi^+_\alpha(w g v_\alpha) \neq 0$.
By the Bruhat decomposition, we may write $w g = u^- w' p^-$, for some $u^- \in \Ua^-$, $p^- \in \Pa^-$, and $w'$ in the $\QQ$-Weyl group.
Then, since $\pi^+_\alpha(\Ua^- v) = \pi^+_\alpha(v)$, for all~$v$, we have
\[ \pi^+_\alpha(w' p^- v_\alpha)= \pi^+_\alpha(u^- w' p^- v_\alpha)= \pi^+_\alpha(w g v_\alpha)	\neq 0.
\]
However, we also know that $p^- v_\alpha$ is in the lowest weight space. Therefore, $w'$ must map the lowest weight to the highest weight. Since $-\lambda_\alpha$ is the lowest weight and $w(-\lambda_\alpha)$ is the highest weight, this means $w'(-\lambda_\alpha) = w(-\lambda_\alpha)$. Since this is true for all~$\alpha$ (and $\{-\lambda_\alpha\}$ spans the dual of~$\Ta$), we conclude that $w' = w$. So
	\[ g
	= w^{-1} u^- w' p^-
	= w^{-1} u^- w p^-
	\in w^{-1} \Ua^- w \Pa^-
	= \Ua^+ \Pa^-
	= \Omega
	, \]
as desired.
\end{proof}

The following easy consequence must also be well known.
It relies on our assumption in Notation~\ref{GaNotation} that $\Ga$ is simply connected.

\begin{cor}[cf.\ {\cite[Lemma~4.11(ii)]{MR1828742}}] \label{F[Omega]}
We have
	\[\QQ[\Omega] = \QQ[\Ga] \left[ \frac{1}{\prod_\alpha \omega_\alpha} \right] .\]
\end{cor}

\begin{proof}
Since $\Ga$ is semisimple and simply connected, it is well known that $\QQ[\Ga]$ is a Unique Factorization Domain (cf.\ \cite[Proposition~3.4]{Rosengarten-Picard}, or see \cite[Corollary~on p.~303]{Popov-Picard} for an explicit statement in the case where $\QQ$ is replaced with an algebraically closed field). Therefore, the desired conclusion is an immediate consequence of~Lemma~\ref{OmegaIdeal}.
\end{proof}

\begin{cor}[cf.\ proof of {\cite[Lemma~4.11(iii)]{MR1828742}}] \label{ubdd}
For $\gamma \in \Gamma \cap \Omega$, we have $\| p^-(\gamma) \| \prec \| \gamma \|$
and $\|u^+_i(\gamma) \| \prec \| \gamma \|$ $($for all~$i)$.
\end{cor}

\begin{proof}Since $g \mapsto p^-(g)$ and $g \mapsto u^+_i(g)$ are regular functions on~$\Omega$, and $\omega_\alpha(\Gamma) \subseteq \ZZ$ (for all~$\alpha$), the desired bounds are immediate from Corollary~\ref{F[Omega]} (and Lemma~\ref{OmegaIdeal}).
\end{proof}

\section{Proof of the main result} \label{MainProofSect}

We will see that Theorem~\ref{QIBddGenByRank1} is an easy consequence of the following lemma, which is a slight modification of a result of Lubotzky--Mozes--Raghunathan (and is proved by essentially the same argument).

\begin{lem}[cf.\ {\cite[paragraphs~4.19, 4.21, and 4.22]{MR1828742}}] \label{OneStep}
Let $i \in \{1,2,\ldots,k\}$, and let $F_i$ be a~finite subset of~$\Ua^+(\QQ)$.
There is a finite subset $F_{i+1}$ of~$\Ua^+(\QQ)$, such that if $\gamma \in \Omega(\ZZ)$, and $\la{u^+_i}(\gamma) \in F_i$, then there exists $x_i \in {}^{\la{u^+_i}(\gamma)}\Ga_i \cap \Ga(\ZZ)$, such that
\begin{enumerate}\itemsep=0pt
\item[$1)$] %\label{OneStep-u}
	$\la{u^+_{i+1}}(x_i \gamma) \in F_{i+1}$, and
\item[$2)$] %\label{OneStep-norm}
	$\| x_i \| \prec \| \gamma \|$.
\end{enumerate}
\end{lem}

\begin{proof}For convenience, let
	\[ u = u^+(\gamma), \qquad \overleftarrow{u} = \la{u^+_i}(\gamma), \qquad \widehat{u} = u^+_i(\gamma), \ \qquad \overrightarrow{u} = \ra{u^+_i}(\gamma), \qquad \text{and}\qquad p = p^-(\gamma) .\]
Since $\overleftarrow{u} \in \Ga(\QQ)$, we may fix a finite-index subgroup~$\Gamma_i$ of $\Ga_i(\ZZ)$, such that ${}^{\overleftarrow{u}}\Gamma_i \subseteq \Ga(\ZZ)$.

We claim there is a finite subset $F^0$ of $\Ua^+_i(\QQ)$, such that $\Ga_i(\QQ) = \Gamma_i F^0 \Pa_i^-(\QQ)$. (This is implicit in the proof of \cite[Lemma~4.19]{MR1828742}, but, for completeness, we record the argument.) It is well known that there is a finite subset~$F^0$ of $\Ga_i(\QQ)$, such that $\Ga_i(\QQ) = \Gamma_i F^0 \Pa_i^-(\QQ)$ \cite[Proposition~15.6]{MR0244260}. There is no harm in multiplying elements of~$F^0$ on the left by elements of~$\Gamma_i$ and on the right by elements of~$\Pa_i^-(\QQ)$; we show that, by doing this, we may assume $F^0 \subseteq \Ua^+_i(\QQ)$.
Note that the set $\Ua^+_i \Pa_i^-$ is Zariski open in~$\Ga_i$ (``big cell'').
Also, for any $f \in F^0$, the Borel Density Theorem \cite[Theorem~4.10, p.~205] {MR1278263} implies that the set $\Gamma_i f$ is Zariski dense in~$\Ga_i$ (because $\Ga_i(\RR)/\Gamma_i$ has finite volume \cite[Theorem~4.13, p.~213]{MR1278263}). These two sets must therefore intersect. Since the map $\Ua^+_i \times \Pa_i^- \to \Ua^+_i \Pa_i^-$ is a biregular $\QQ$-isomorphism \cite[Proposition~3.24, p.~84]{BorelTits-GrpRed}, this implies $(\Gamma_i f) \cap \bigl(\Ua^+_i(\QQ) \Pa_i^-(\QQ) \bigr) \neq \varnothing$. Hence, after multiplying $f$ on the left by an appropriate element of~$\Gamma_i$ and on the right by an appropriate element of $\Pa_i^-(\QQ)$, we may assume $f \in \Ua^+_i$. This completes the proof of the claim.

Now, the claim tells us that we may write
	\[ \widehat{u} = x f q \qquad \text{with} \quad x \in \Gamma_i, \quad f \in F^0 \subset \Ua^+_i(\QQ), \quad \text{and} \quad q \in \Pa_i^-(\QQ) .\]
Note that
	\begin{gather*}
	 {}^{\overleftarrow{u}} \! \! x^{-1} u
	 = {}^{\overleftarrow{u}} \! \! x^{-1} \overleftarrow{u} \widehat{u} \overrightarrow{u}
	= \overleftarrow{u} x^{-1} (x f q) \overrightarrow{u}
	= \overleftarrow{u} f {}^q \overrightarrow{u} q
	. \end{gather*}
It follows from the definitions that $\ra{\Ua_i}$ is normalized by $\Pa_i^-$ (indeed, it normalized by $\Ga_i$), so we have
	\begin{gather} \label{quplus}
	{}^q \overrightarrow{u}
	\in \ra{\Ua_i}
	= \ra{\Ua_i^+} \big(\ra{\Ua_i} \cap \unip \Pa^-\big)
	\subseteq \ra{\Ua_i^+} \Pa^-
	. \end{gather}
We also know $q \in \Pa_i^- \subseteq \Pa^-$. Therefore
	\begin{gather} \label{usgamma}
	{}^{\overleftarrow{u}} \! \! x^{-1} \gamma
	= {}^{\overleftarrow{u}} \! \! x^{-1} (u p)
	= \big({}^{\overleftarrow{u}} \! \! x^{-1} u\big) p
	= \big(\overleftarrow{u} f {}^q \overrightarrow{u} q\big) p
	\in \overleftarrow{u} f \big(\ra{\Ua_i^+} \Pa^-\big) \Pa^- \Pa^-
	= \overleftarrow{u} f \ra{\Ua_i^+} \Pa^-
	. \end{gather}
Since $\overleftarrow{u} f \in \la{\Ua^+_i}(\QQ) \cdot \Ua^+_i(\QQ) = \la{\Ua^+_{i+1}}(\QQ)$, this implies
	\[ \la{u^+_{i+1}} \big( {}^{\overleftarrow{u}} \! \! x^{-1} \gamma\big) = \overleftarrow{u} f .\]
Now, let $x_i = {}^{\overleftarrow{u}} \! \! x^{-1}$ and $F_{i+1} = F_i F^0$, so we have
	\begin{gather*}
	\la{u^+_{i+1}} ( x_i \gamma)
	= \la{u^+_{i+1}} \big( {}^{\overleftarrow{u}} \! \! x^{-1} \gamma\big)
	= \overleftarrow{u} f
	\in F_i F^0
	= F_{i+1}
	. \end{gather*}
This establishes (1).

All that remains is to show we can choose~$x$ so that (2) holds.
To this end, let $\widehat{U}_\ZZ = {}^{\overleftarrow{u}}\Ua^+_i \cap \Ga(\ZZ)$, so $\widehat{U}_\ZZ$ is a cocompact lattice in~${}^{\overleftarrow{u}}\Ua^+_i(\RR)$ \cite[Theorem~4.12, p.~210]{MR1278263}. Then Corollary~\ref{ubdd} implies that we may assume
	$\| u^+_i(\gamma) \| = O(1) $,
after multiplying~$\gamma$ on the left by an element~$\widehat{v}$ of~$\widehat{U}_\ZZ$, such that $\| \widehat{v} \| \prec \| \gamma \|$. This means $\| \widehat u \| = O(1)$.

Let $\Pa_i^- = \Ma_i \Aa_i \Na_i$ be the Langlands decomposition (with $\Aa_i = (\Ta \cap \Ga_i)^\circ$), and write
	\[ q = ma, \qquad \text{with} \quad m \in (\Ma_i \Na_i)(\QQ) \quad \text{and} \quad a \in \Aa_i(\QQ) .\]
Note that $\Ma_i \Na_i$ has no nontrivial $\QQ$-characters, so ${}^{f^{-1}} \Gamma_i \cap (\Ma_i \Na_i)$ is a cocompact lattice in $(\Ma_i \Na_i)(\RR)$ \cite[Theorem~4.12, p.~210]{MR1278263}. Therefore, after multiplying $m$ on the left by an appropriate element of ${}^{f^{-1}} \Gamma_i \cap (\Ma_i \Na_i)$ (and multiplying $x$ on the right by the $f$-conjugate of the inverse of this element), we may assume $\| m \| = O(1)$ (and we still have $\widehat{u} = x f q$). Choose a simple $\QQ$-root~$\alpha$ that is not orthogonal to $\Phi_i$ (so the restriction of~$\lambda_\alpha$ to~$\Aa_i$ has finite kernel). Then
	\begin{gather*}
	\|x\| = \big\| \widehat{u} q^{-1} f^{-1}\big\|
	= \big\| \widehat{u} a^{-1} m^{-1} f^{-1}\big\|\\
\hphantom{\|x\|}{} \le O(1) \cdot \big\|a^{-1}\big\| \cdot O(1) \cdot O(1)
	\prec \|a\|
	\prec \max \bigl( |\omega_\alpha(a)|, \big|\omega_\alpha\big(a^{-1}\big)\big| \bigr)
	. \end{gather*}
Furthermore, we have
	\begin{align*}
	\big|\omega_\alpha\big(a^{-1}\big)\big|
	&=\big|\omega_\alpha\big(q^{-1}\big)\big|
		&& \text{\big($a^{-1} = q^{-1} m \in q^{-1} \Ma_i \Na_i$, and $\Ma_i \Na_i$ fixes~$v_\alpha$\big)}
	\\&=\big|\omega_\alpha\big(\widehat{u} q^{-1}\big)\big|
		&& \text{\big($\widehat{u} \in \Ua^+$ fixes $v_\alpha^*$\big)}
	\\&=|\omega_\alpha(x f)|
		&& \text{\big($\widehat{u} = x f q \ \Rightarrow \ \widehat{u} q^{-1} = x f$\big)}
	\\& > \frac{1}{O(1)}	,
\end{align*}
 because the denominator of $\omega_\alpha(x f)$ is bounded, since $f$ is in a finite subset of $\Ga(\QQ)$ and $x \in \Ga(\ZZ)$ (and $\omega_\alpha(x f)\neq 0$, since $x f = \widehat{u} q^{-1} \in \Ua^+ \Pa^- = \Omega$). So
	\[ \max \bigl( |\omega_\alpha(a)|, \big|\omega_\alpha\big(a^{-1}\big)\big| \bigr)
	= \max \bigl( \big|\omega_\alpha\big(a^{-1}\big)\big|^{-1}, \big|\omega_\alpha\big(a^{-1}\big)\big| \bigr)
	= \max \bigl( O(1), \big|\omega_\alpha\big(a^{-1}\big)\big| \bigr) .\]
However, we also have
	\begin{align*}
	\big| \omega_\alpha \big( {}^{\overleftarrow{u}} \! \! x^{-1} \gamma\big) \big|
	&= \bigl| \omega_\alpha \bigl( \overleftarrow{u} f {}^q \overrightarrow{u} q p \bigr) \bigr|
		&& \text{(delete last two parts of \eqref{usgamma})}
	\\&\in \bigl| \omega_\alpha \bigl( \overleftarrow{u} f \big(\Ua^+ \unip \Pa^-\big) q p \bigr) \bigr|
		&& \text{(see \eqref{quplus})}
	\\&= \bigl| \omega_\alpha \bigl( \Ua^+ (\unip \Pa^-) m a p \bigr) \bigr|
		&& \text{\big($\overleftarrow{u},f \in \Ua^+$ and $q = ma$\big)}
	\\&= \bigl| \omega_\alpha \bigl( a \bigr) \bigr| \cdot \bigl| \omega_\alpha \bigl( p \bigr) \bigr|
		&& \begin{pmatrix} \text{$\Ua^+$ fixes $v_\alpha^*$, and } \\
		\text{$\omega_\alpha|_{\Pa^-}$ is a homomorphism} \end{pmatrix}.
	\end{align*}
Since this is obviously nonzero, and ${}^{\overleftarrow{u}} \! \! x^{-1} \gamma \in \Gamma$, we conclude from \eqref{omega(Gamma)} that $ |\omega_\alpha ( a ) | \cdot |\omega_\alpha ( p ) | \ge 1$. Therefore
	\[ \big|\omega_\alpha\big(a^{-1}\big)\big|
	= \frac{1}{|\omega_\alpha(a)|}
	\le |\omega_\alpha ( p ) |
	\prec \| p\|
	\prec \| \gamma\| .\]
By stringing some of these inequalities together (and noting that $O(1) \prec \| \gamma\|$), we see that
	\[ \| x\|
	 \prec \max \bigl( |\omega_\alpha(a)|, \big|\omega_\alpha\big(a^{-1}\big)\big| \bigr)
	 = \max \bigl( O(1), \big|\omega_\alpha\big(a^{-1}\big)\big| \bigr)
	 \prec \| \gamma\|
	 , \]
which establishes~(2).
\end{proof}

It is now easy to prove the main theorem by using the argument (due to Lubotzky--Mozes--Raghunathan \cite[Section~4]{MR1828742}) that is described in Section~\ref{SLnZSect}:

\begin{proof}[\bf Proof of Theorem~\ref{QIBddGenByRank1}]
Recall that we may assume $\Gamma = \Ga(\ZZ)$ (see Notation~\ref{QIBddGenNotation}(1)).
Thus, given $\gamma \in \Ga(\ZZ)$, we wish to show that $\gamma$ can be written as an appropriate product $\gamma = x_1 x_2 \cdots x_r$. Since $\Ga(\ZZ)$ is Zariski dense in~$\Ga$ \cite[Theorem~4.10, p.~205]{MR1278263} and $\Omega$ is Zariski open in~$\Ga$, we have $\Ga(\ZZ) \Omega = \Ga$. Then the ascending chain condition on Zariski-open subsets implies there is a~finite subset $\Gamma_1$ of $\Ga(\ZZ)$, such that $\Gamma_1 \Omega = \Ga$. This implies $\gamma \in \Gamma_1 \Omega(\ZZ)$, so, by adding the elements of~$\Gamma_1$ to~$\Gamma_0$, we may assume $\gamma \in \Omega(\ZZ)$.

Let $F_1 = \{e\} \subset \Ua^+(\QQ)$. Repeated application of Lemma~\ref{OneStep} yields elements $x_1,x_2,\ldots,x_k$ of~$\Ga(\ZZ)$ and finite subsets $F_2, F_3,\ldots, F_{k+1}$ of $\Ua^+(\QQ)$, such that, for each $i = 1,2,\ldots,k$, we have
\begin{enumerate}\itemsep=0pt
\item[1)] $\la{u^+_{k+1}}(x_i x_{i-1} \cdots x_1 \gamma) \in F_{i+1}$,
\item[2)] $x_i \in {}^f \Ga_i \cap \Ga(\ZZ)$, for some $f \in F_{i}$, and
\item[3)] $\| x_i \| \prec \gamma$ for all~$i$.
	\end{enumerate}
Letting $\gamma_1 = x_k x_{k-1} \cdots x_1 \gamma$, we see (from the case where $i = k$) that $u^+(\gamma_1) = u^+_{k+1}(\gamma_1)$ is in a finite set, so there exists $x$ in a finite subset~$\Gamma_0$ of $\Ga(\ZZ)$, such that $u^+(x\gamma_1)$ is trivial, which means $x \gamma_1 \in \Pa^-$.

Write $\Pa^- = \Na \Ma \Aa$ (Langlands decomposition). Multiplying $x \gamma_1$ on the left by an element of $\Na(\ZZ)$ (and an element of a finite set) yields $\gamma_2 \in \Ma(\ZZ)$. (Note that any element of $\Na(\ZZ)$ is the product of a bounded number of root elements of controlled norm, and each root element is contained in a standard $\QQ$-rank-one subgroup.) Then Lemma~\ref{LeviOK} (and induction on $\rank_\QQ \Ga$) completes the proof (under Assumption~\ref{AbsSimple} that $\Ga$ is absolutely almost simple).
\end{proof}

\begin{rem} \label{ChangeForK}To eliminate the simplifying assumption Assumption~\ref{AbsSimple} that $\Ga$ is absolutely almost simple, one should modify the above argument to adhere a bit more closely to~\cite{MR1828742}, where there is no such assumption.
Begin by noting that $\Ga$ is the restriction of scalars of some absolutely almost-simple group over a number field~$K$ \cite[paragraph~6.21(ii), p.~113]{BorelTits-GrpRed}. Then replace $\QQ$ and~$\ZZ$ with $K$ and~$\ints$ (the ring of integers of~$K$), and replace the norm $\|\gamma\|$ with
	\[ \|\gamma\|_\infty = \max_{v \in S_\infty} \| \gamma\|_v , \qquad \text{where} \quad \text{$S_\infty$ is the set of all archimedean places of~$K$.} \]
Also, as in \cite[p.~35]{MR1828742}, let
	$|t|^* = \prod\limits_{v \in S_\infty} |t|_v$, for $t \in K$.
(When $v$ is a complex place, $|t|_v$ denotes $t \cdot \overline{t}$, not $\sqrt{t \cdot \overline{t}}$.)
Thus, $|t|^* \in \ZZ$ for $t \in \ints$.

In the proof of Corollary~\ref{ubdd}, note that $\omega_\alpha(\gamma) \in \ints$ (and $\omega_\alpha(\gamma) \neq 0$), so $|\omega_\alpha(\gamma)|^* \ge 1$; this implies $\max\limits_{v \in S_\infty} |1/\omega_\alpha(\gamma) |_v \prec \max\limits_{v \in S_\infty} | \omega_\alpha(\gamma) |_v$.

In the proof of Lemma~\ref{OneStep}, instead of cocompact lattices in ${}^{\overleftarrow{u}}\Ua^+_i(\RR)$ and $(\Ma_i \Na_i)(\RR)$, we have cocompact lattices in $\prod\limits_{v \in S_\infty} {}^{\overleftarrow{u}}\Ua^+_i(K_v)$ and $(\Pa_i^-)_{S_\infty}^{(1)}$ (see \cite[Theorem~5.7(2), p.~268]{MR1278263} and cf.\ \cite[Theorem~5.7, p.~264]{MR1278263}). With this in mind, we write a decomposition $q = ma$, where $m \in (\Pa_i^-)_{S_\infty}^{(1)}$ and where, after identifying the $1$-dimensional $K$-split torus~$\Aa_i$ with~$\GG_m$, we have $a \in \GG_m(\QQ) = \QQ^\times$. Note that $\omega_\alpha(a) \in \QQ$ (because every character of~$\GG_m$ is defined over~$\QQ$), so
	$|\omega_\alpha(a)| \asymp |\omega_\alpha(a)|^*$.
Therefore, the above arguments show:
	\begin{gather*}
	\|x\|_\infty \prec \|a\| \prec \max \bigl( |\omega_\alpha(a)|, \big|\omega_\alpha\big(a^{-1}\big)\big| \bigr) , \\
	\big|\omega_\alpha\big(a^{-1}\big)\big| \succ \big|\omega_\alpha\big(a^{-1}\big)\big|^* = |\omega_\alpha(x f)|^* > \frac{1}{O(1)} , \\
	 \max \bigl( |\omega_\alpha(a)|, \big|\omega_\alpha\big(a^{-1}\big)\big|\bigr) = \max \bigl( O(1), \big|\omega_\alpha\big(a^{-1}\big)\big| \bigr) , \\
	 \big|\omega_\alpha\big(a^{-1}\big)\big| \prec \big|\omega_\alpha\big(a^{-1}\big)\big| ^* \prec \| \gamma \|_\infty
	. \end{gather*}
These yield the required bound $\|x\|_\infty \prec \| \gamma \|_\infty$.
\end{rem}

\section[Generalization to $S$-arithmetic groups]{Generalization to $\boldsymbol{S}$-arithmetic groups}\label{SarithSect}

Theorem~\ref{QIBddGenByRank1} is stated only for arithmetic subgroups of~$\Ga$, but it generalizes in a natural way to the $S$-arithmetic setting:

\begin{prop}Let $\Ga$ be an isotropic, almost-simple algebraic group over~$\QQ$, let~$S$ be a~finite set of valuations of~$\QQ$ that includes the archimedean valuation~$\infty$, and let~$\Gamma$ be an $S$-arithmetic subgroup of~$\Ga$. Then $\Gamma$ is quasi-isometrically boundedly generated by standard $\QQ$-rank-$1$ subgroups.
\end{prop}

In fact, the methods of Lubotzky--Mozes--Raghunathan \cite{MR1828742} also apply in the following more general situation:

\begin{prop} \label{GeneralSArith}Let $\Ga$ be an isotropic, absolutely almost-simple algebraic group over a number field~$K$, let $S$ be a finite set of valuations of~$K$ that includes all of the archimedean valuations, and let $\Gamma$ be an $S$-arithmetic subgroup of~$\Ga$. Then $\Gamma$ is quasi-isometrically boundedly generated by standard $K$-rank-$1$ subgroups.
\end{prop}

The statement of this proposition assumes that Definition~\ref{QIBddGenDefn} has been adapted to this situation by replacing $\QQ$ with~$K$, replacing $\Ga(\ZZ)$ with $\Ga(\ints_S)$, where $\ints_S$ is the ring of $S$-integers of~$K$, and replacing $\|\gamma\|$ with
	$\|\gamma\|_S = \max_{v \in S} \| \gamma \|_v$.

\begin{proof}[Sketch of the proof of Proposition~\ref{GeneralSArith}]
The argument is essentially the same as the modification described in Remark~\ref{ChangeForK}, except that it uses $S$ and the ring~$\ints_S$ of $S$-integers of~$K$, in the place of $S_\infty$ and~$\ints$. In particular, we let $|t|^* = \prod\limits_{v \in S} |t|_v$, so $|t|^* \in \ZZ$ for $t \in \ints_S$.

However, $x$ is taken to be an element of $\Ga(\ints)$; this is a crucial instance where $\ints$ is not replaced with~$\ints_S$. This implies that $\|x\|_v$ and $\big\|x^{-1}\big\|_v$ are bounded, for all nonarchimedean $v \in S$. Hence, although the goal is to show $\|x\|_S \prec \| \gamma \|_S$, it suffices to show $\|x\|_\infty \prec \| \gamma \|_S$.

Note that, after multiplying $a$ by an element of $\Aa_i(\ints_S) = (\ints_S)^\times$ (and multiplying $m$ by the inverse of this element), we may assume that $\|a\|_v$ and $\big\|a^{-1}\big\|_v$ are bounded, for all nonarchimedean $v \in S$. Also, since $a \in \QQ^\times$, we have $|\omega_\alpha(a)|_{v_1} = |\omega_\alpha(a)|_{v_2}$ for all archimedean~$v_1$,~$v_2$. Therefore, we may write $ |\omega_\alpha(a)|$, omitting the subscript. And then we have $|\omega_\alpha(a)| \asymp |\omega_\alpha(a)|^*$.
It is now easy to generalize the argument in the final paragraph of Remark~\ref{ChangeForK}.
\end{proof}

\begin{rem}The results of Lubotzky--Mozes--Raghunathan~\cite{MR1828742} are proved over all global fields, not only those of characteristic zero, but I do not have the expertise to speculate on whether there is a similar generalization of Proposition~\ref{GeneralSArith}.
\end{rem}

\section{Semisimple Lie groups with infinite center} \label{InfiniteCenterSect}

\begin{defn} \label{StandardQInfiniteCenterDefn} Let $\cover\Gamma$ be a noncocompact, irreducible lattice in a connected, semisimple Lie group~$\cover G$ that has no compact factors. (The center of~$\cover G$ may be infinite, but it is a well-known and easy consequence of the Borel density theorem that~$\cover\Gamma$ contains a finite-index subgroup of~$Z\big(\cover G\big)$, so $\Ad_{\cover G} \cover \Gamma$ is a lattice in~$\Ad \cover G$.)
\begin{enumerate}\itemsep=0pt
\item If $\rank_\RR \cover G = 1$, then $\cover G$ is the only \emph{standard $\QQ$-rank-$1$ subgroup} of~$\cover G$.
	\item If $\rank_\RR \cover G \ge 2$, then the Margulis arithmeticity theorem implies that (up to finite index)~$\Ad \cover G$ can be viewed as the real points of an algebraic group over~$\QQ$, in such a way that~$\Ad_{\cover G} \cover \Gamma$ is commensurable to $\big(\Ad \cover G\big)_\ZZ$.
Therefore, we can speak of standard $\QQ$-rank-1 subgroups of $\Ad \cover G$. For any such subgroup~$L$ of $\Ad \cover G$, the identity component of the inverse image $\Ad_{\cover G}^{-1} L$ is a \emph{standard $\QQ$-rank-$1$ subgroup} of~$\cover G$.
	\end{enumerate}
\end{defn}

The main theorem~\ref{QIBddGenByRank1} has the following consequence:

\begin{cor} \label{InfCenterBddGen}If $\cover\Gamma$ is a noncocompact, irreducible lattice in a connected, semisimple Lie group~$\cover G$ that has no compact factors, then $\cover\Gamma$ is quasi-isometrically boundedly generated by standard $\QQ$-rank-$1$ subgroups.

More precisely, there exist constants $r = r\big(\cover G,\cover\Gamma\big) \in \NN$ and $C = C\big(\cover G,\cover\Gamma\big) \in \RR^+$, a finite subset $\cover\Gamma_0 = \cover\Gamma_0\big(\cover G,\cover\Gamma\big)$ of~$\cover\Gamma$, and a finite collection $\subgroups = \subgroups\big(\cover G,\cover\Gamma\big)$ of standard $\QQ$-rank-$1$ subgroups of~$\cover G$, such that every element~$\gamma$ of~$\cover\Gamma$ can be written in the form $\gamma = x_1 x_2 \cdots x_r$, where, for each~$i$, we have either:
		\begin{enumerate}\itemsep=0pt
		\item[$1)$] %\label{InfCenterBddGen-product-sinL}
		$x_i \in L \cap \cover\Gamma$, for some $L \in \subgroups$, and $\ell_{\cover\Gamma}(x_i) \le C \ell_{\cover\Gamma}(\gamma)$,
		or
		\item[$2)$] $x_i \in \cover\Gamma_0$.
		\end{enumerate}
\end{cor}

If the center of~$\cover G$ is finite, then this is simply Corollary~\ref{QIBddGenLie} (and is immediate from Theo\-rem~\ref{QIBddGenByRank1}). In the case where the center is infinite, the proof is completed by the following observation:

\begin{lem} \label{QContainsZ} There is a finite set~$\subgroups$ of standard $\QQ$-rank-$1$ subgroups of~$\cover G$, such that $\prod\limits_{L \in \subgroups} \bigl( L \cap Z\big(\cover G\big) \bigr)$ has finite index in $Z\big(\cover G\big)$.
\end{lem}

To prove this lemma, let us first make another definition:

\begin{defn}A closed, connected, almost-simple subgroup~$L$ of a connected, semisimple Lie group~$\cover G$ is a \emph{standard $\RR$-rank-$1$ subgroup} if there is a root~$\alpha$ of some maximal $\RR$-split torus~$S$ of~$G$, such that the Lie algebra of~$L$ is generated by the root spaces $\Lie U_\beta$ for $\beta \in \big\{{\pm} \alpha, \pm 2\alpha, \pm \frac{1}{2} \alpha \big\}$. (This implies that $\rank_\RR L = 1$.)
\end{defn}

If we choose a maximal $\RR$-split torus that contains a maximal $\QQ$-split torus, then every real root space is contained in either a $\QQ$-root space or the $\QQ$-anisotropic kernel. Since all maximal $\RR$-split tori are conjugate, we conclude that every standard $\RR$-rank-1 subgroup is contained in a~conjugate of a standard $\QQ$-rank-1 subgroup. Hence, it suffices to establish the following result (which is probably known, but we do not have a reference):

\begin{lem}There is a finite set~$\subgroups$ of standard $\RR$-rank-$1$ subgroups of~$\cover G$, such that
	$\prod\limits_{L \in \subgroups} \bigl( L \cap Z\big(\cover G\big) \bigr)$
has finite index in $Z\big(\cover G\big)$.
\end{lem}

\begin{proof}There is nothing to prove unless $Z\big(\cover G\big)$ is infinite.
By treating each simple factor of~$\cover G$ individually, we may assume that $\cover G$ is simple modulo its center. Then it is well known (and can also be seen from Table~\ref{InfiniteFundGrpTable}) that $Z\big(\cover G\big)$ is virtually cyclic, so:
	\[ \text{\it It suffices to show that some standard $\RR$-rank-$1$ subgroup of~$\cover G$ has infinite center.} \]
Obviously, we may assume $\rank_\RR \cover G \ge 2$.
It is well known that (after passing to a finite cover) $\cover G$ is the universal cover of one of the groups~$G$ listed in Table~\ref{InfiniteFundGrpTable}.

	\begin{table}[ht]
	\centerline{$\begin{array}{|c|c|c|}
	\noalign{\hrule}
	G & \text{maximal compact subgroup $K$} & \text{restrictions} \vphantom{\Big|}\\
	\noalign{\hrule}
	\Sp(2n,\RR) & \Unitary(n) & n \ge 2 \vphantom{\Big|}\\
	\SO(2,n)^\circ & \SO(2) \times \SO(n) & n \ge 4 \vphantom{\Big|}\\
	\SU(m,n) & S \bigl( \Unitary(m) \times \Unitary(n) \bigr) & n \ge m \ge 2 \vphantom{\Big|}\\
	\SO^*(2n) = \SO(n,\HH) & \Unitary(n) & n \ge 4 \vphantom{\Big|}\\
	\text{E III} = E_6^{-14} = {}^2E_{6,2}^{16'} & \SO(2) \times \SO(10) & \vphantom{\Big|}\\
	\text{E VII} = E_7^{-25} = E_{7,3}^{28} & \SO(2) \times E_6\vphantom{\Big|} & \\
	\noalign{\hrule}
	\end{array}$}
	\caption{Simple Lie groups with infinite fundamental group and real rank at least two. (The group $\SU(1,n)$ of real rank one is deleted from the well-known list in \cite[p.~518]{Helgason-DiffLieSymm} and \cite[(7.147), p.~513]{Knapp-BeyondIntro}.)} \label{InfiniteFundGrpTable}
	\end{table}

From Table~\ref{InfiniteFundGrpTable}, we see that $\rank K = \rank G$, so $G$ has a compact maximal torus~$T$. Consider the root system of~$\Lie G_\CC$ with respect to this maximal torus. For each root~$\alpha$,
the values of~$\alpha$ on the Lie algebra of~$T$ are purely imaginary (because $T$ is compact), so complex conjugation sends $\alpha$ to~$-\alpha$. This implies (as is well known) that the subgroup generated by $\alpha$ and~$-\alpha$ is invariant under complex conjugation, and is therefore defined over~$\RR$. We let $G_\alpha$ be the resulting connected subgroup of~$G$, so $G_\alpha$ is locally isomorphic to either $\SL(2,\RR)$ or $\SU(2)$.

Fix a maximal compact subgroup~$K$ that contains~$T$.
The Vogan diagram of~$G$ is listed in Table~\ref{VoganDiagFig}. (The black root is the unique simple root~$\beta$, such that $G_\beta$ is not compact, i.e., such that $G_\beta \not\subseteq K$. Because it will be useful later in the proof, the minimal root~$\mu$ has been added to each picture, even though it is not a simple root and is therefore not actually part of the Vogan diagram.)
The white roots form a Dynkin diagram whose rank is one less than $\rank_\CC G$; more precisely, they form the Dynkin diagram of the semisimple part of~$K$. Hence (as is well known),
	\[\langle G_\omega \,|\, \text{$\omega$ is a white simple root} \rangle \]
is the entire semisimple part of~$K$. Therefore, any root~$\alpha$ that is not in the subalgebra generated by the white simple roots must be noncompact. (This means $\theta(x) = -x$ for all~$x$ in $\Lie U_\alpha^\CC$, where~$\theta$ is the Cartan involution determined by~$K$.) In particular, since the minimal root~$\mu$ is not a~simple root, we know that $G_\mu$ is noncompact, and therefore contains a nontrivial $\RR$-split torus~$S_\mu$ \cite[pp.~390--391]{Knapp-BeyondIntro}.

	\begin{table}[bt]
	\centerline{$\begin{array}{|c|c|}
	\noalign{\hrule}
	G & \text{Vogan diagram (and minimal root~$\mu$)} \vphantom{\Big|}\\
	\noalign{\hrule}
	\raise6mm\hbox{$\begin{matrix} \SO(2,n) \\ \text{($n$ even)} \end{matrix}$}
		 & \vphantom{\vbox to 1.75cm{\vss}} \includegraphics{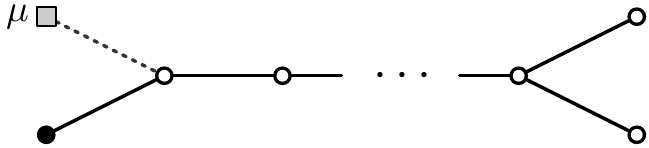} \\[2mm]
	\raise7mm\hbox{$\begin{matrix} \SO(2,n) \\ \text{($n$ odd)} \end{matrix}$}
		 & \includegraphics{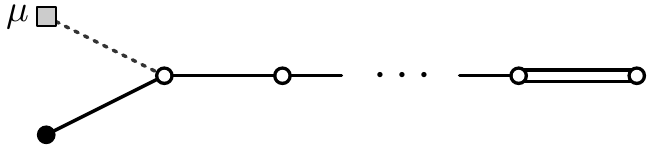} \\[2mm]
	\raise0.5mm\hbox{$\Sp(2n,\RR)$} & \includegraphics{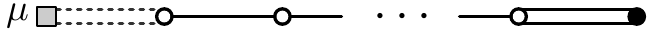} \\[2mm]
	\raise7mm\hbox{$\SU(m,n)$} & \includegraphics{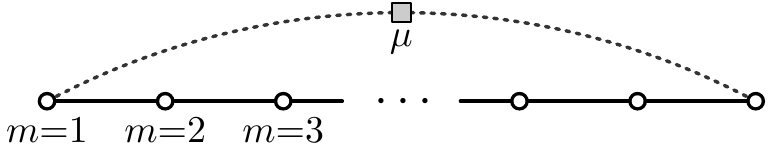} \\[2mm]
	\raise6mm\hbox{$\SO(n,\HH)$} & \includegraphics{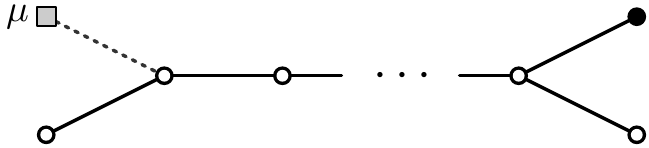} \\[-2mm]
	\raise10mm\hbox{E III} & \includegraphics{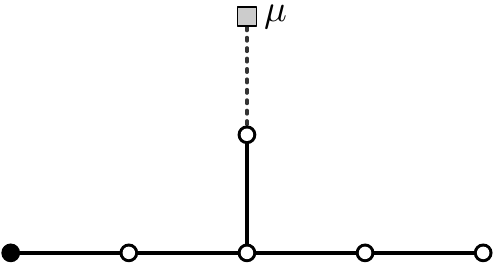} \\[2mm]
	\raise5mm\hbox{E VII} & \includegraphics{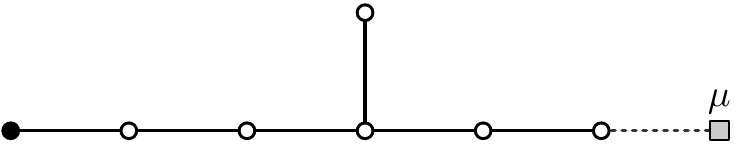} \\[2mm]
	\noalign{\hrule}
	\end{array}$}
	\caption{Vogan diagram \cite[Figs.~6.1 and~6.2, pp.~414 and~416]{Knapp-BeyondIntro} and minimal root~$\mu$ \cite[Table~I, p.~53]{Tits-Classif} of each simple Lie group with infinite fundamental group. For $\SU(m,n)$, the $m$th vertex from the left is black.} \label{VoganDiagFig}
	\end{table}

A quick look at Table~\ref{VoganDiagFig} shows that the black root~$\beta$ is not joined by an edge to the minimal root~$\mu$. (For $\SU(m,n)$, note that $m \le n$, so the black vertex is in the left half, but $m > 1$, so the black vertex is not the vertex at far left.) Therefore $G_\beta$ is centralized by $G_\mu$, so it is centralized by a nontrivial $\RR$-split torus (namely~$S_\mu$). This implies there is a proper parabolic $\RR$-subgroup~$P$ of~$G$, with Langlands decomposition $P = MAN$, such that $G_\beta \subseteq M$. Let $M'$ be the almost simple factor of~$M^\circ$ that contains~$G_\beta$.

Let $\Delta$ be the set of simple roots. (So $\Delta$ is the set of nodes of the Vogan diagram.) It is clear that $T = \langle (G_\delta \cap T)^\circ \,|\, \delta \in \Delta \rangle$.
Also, since $Z(K)^\circ \subset T$, we know that the lift of~$T$ to $\cover G$ contains an infinite subgroup of $Z\big(\cover G\big)$. Therefore, this must also be true of $(G_\delta \cap T)^\circ$, for some $\delta \in \Delta$. However, if $\delta$ is a white root, then $G_\delta$ is a compact simple group, so its universal cover has finite center. Hence, the lift of $G_\beta$ (where $\beta$ is the black simple root) must contain an infinite subgroup of $Z\big(\cover G\big)$. Since $G_\beta \subset M'$, we conclude that the lift of~$M'$ contains a subgroup of $Z\big(\cover G\big)$ that is infinite, and therefore has finite index.

By induction on $\dim G$, there is a standard $\RR$-rank-1 subgroup~$L$ of~$M'$, such that the lift of~$L$ to the universal cover~$\cover{M'}$ of~$M'$ contains a finite-index subgroup of $Z\big(\cover{M'}\big)$. Then the lift of~$L$ to~$\cover G$ contains a finite-index subgroup of $Z\big(\cover G\big)$.
Furthermore, it is easy to see that~$L$ is a~standard $\RR$-rank-1 subgroup of~$G$. (If $A'$ is a maximal $\RR$-split torus of~$M$, then each root space of~$A'$ in the Lie algebra of~$M$ is also a root space of the maximal $\RR$-split torus $A A'$ of~$G$.)
\end{proof}

\subsection*{Acknowledgements}
I thank A.~Brown, D.~Fisher, and S.~Hurtado for suggesting this problem, and for their encouragement as I worked toward a solution. Extra thanks are due to D.~Fisher for suggesting the generalization to groups with infinite center that is presented in Section~\ref{InfiniteCenterSect}.

\pdfbookmark[1]{References}{ref}
\LastPageEnding

\end{document}